\documentclass[letterpaper, 10 pt, conference]{ieeeconf}
\IEEEoverridecommandlockouts
\usepackage{cite}
\usepackage{url}
\usepackage{amsmath,amssymb,amsfonts}
\usepackage{multirow}
\usepackage{subfigure}

\usepackage{graphicx}
\usepackage{textcomp}
\usepackage{epsfig}

\usepackage{epstopdf}
\usepackage[ruled]{algorithm}
\usepackage{algorithmic}
\usepackage{enumerate}
\usepackage{float}
\usepackage{graphics,graphicx}
\usepackage{subfigure}

\usepackage{array}
\usepackage{mathtools}
\allowdisplaybreaks[3]
\usepackage{url}
\DeclareGraphicsExtensions{.png}
\graphicspath{{./Pics/}}
\usepackage{color}
\usepackage{xspace}

\usepackage{hyperref}
\hypersetup{
    colorlinks=true,
    linkcolor=blue,
    filecolor=gray,      
    urlcolor=blue,
    citecolor=blue,
}

\def\BibTeX{{\rm B\kern-.05em{\sc i\kern-.025em b}\kern-.08em
		T\kern-.1667em\lower.7ex\hbox{E}\kern-.125emX}}

\newtheorem{lemma}{Lemma}
\newtheorem{theorem}{Theorem}
\newtheorem{definition}{Definition}

\newtheorem{assumption}{Assumption}

\usepackage{bbm}
\renewcommand{\theequation}{\arabic{equation}}

\begin{document}

\title{\LARGE \bf Communication-Efficient Differentially Private Gradient Tracking for Distributed Optimization via Local Updates}

\author{
Mihitha Maithripala, Chenyang Qiu, and
Zongli Lin
\thanks{Work supported in part by the US National Science Foundation under the grant number CMMI–2243930.}
\thanks{M. Maithripala, C. Qiu, and Z. Lin are with the Charles L. Brown Department of Electrical and Computer Engineering, University of Virginia, Charlottesville, VA 22904, USA.}}

\maketitle
\pagestyle{plain}
\thispagestyle{empty}

\begin{abstract}
This paper studies privacy-preserving distributed optimization using gradient tracking with one communication-free local update between consecutive communication rounds. Since local computation alone does not protect gradient information, we perturb both the primal state and the gradient-tracking direction at communication iterations while keeping local updates noise-free. We establish an aggregate variable for tracking, characterize the limiting consensus point in the presence of added random noise, and prove convergence in mean under certain parameter conditions. We also establish infinite-horizon pure differential privacy for the complete communication transcript under an affine objective adjacency relation and geometrically decaying Laplace perturbations. Simulation results illustrate the advantages of the proposed method in terms of the convergence and optimization accuracy given a privacy budget.

\end{abstract}

\section{Introduction}
\label{sec:introduction}

Distributed optimization allows a network of agents to minimize an aggregate objective through local computation and information exchange with neighboring agents. It provides a natural framework for coordinating networked energy systems~\cite{xu2018dual,liu2024scalable}, processing distributed sensor measurements~\cite{zhu2018mitigating}, and learning from decentralized data~\cite{tsianos2012consensus}. Gradient-tracking methods, including DIGing, combine local descent with an auxiliary variable that tracks aggregate gradient information and achieve geometric convergence under suitable smoothness, strong-convexity, and connectivity conditions~\cite{nedic2017geometric,qu2017harnessing,pu2021distributed}. However, their implementation raises two practical concerns: frequent communication can be costly, and the exchanged states and tracking variables can disclose information about local objectives.

Motivated by the success of local updates in federated learning~\cite{reisizadeh2020fedpaq}, distributed optimization algorithms incorporating local updates~\cite{liu2026guaranteeing,nguyen2023performance, wu2025effectiveness} have been proposed. These algorithms reduce the frequency of information exchange by allowing agents to perform additional local computation between communication rounds. Recent work~\cite{wang2026localupdates} investigated local updates in gradient-tracking based distributed optimization.  Their analysis shows that, with appropriately selected stepsizes, local updates can improve optimization performance per communication round. In particular, their analysis identifies a regime in which one communication-free local update between consecutive communication rounds achieves performance improvement, while repeating the local updates provide no further gain. These findings motivate a simple alternating structure: a communication-free local update followed by an update involving network mixing.

However, whether this communication structure preserves the privacy of sensitive local information remains to be established. This question is important because exchanged states and gradients may reveal sensitive information about participating agents, such as robot positions in multi-robot rendezvous, sensor locations in distributed information fusion, or private training data in distributed learning~\cite{xuan2023gradienttracking}. Indeed, as we will show in this paper, under the stated observation and initialization assumptions, an observer with information of the complete communication sequence can recursively reconstruct the uncommunicated intermediate states and the entire local gradient trajectory of the gradient-tracking method in~\cite{wang2026localupdates} with one local update between consecutive communication rounds. Thus, communication-free local updates alone do not protect gradient information, motivating an explicit privacy-preserving mechanism.

In the absence of local-update mechanisms, various approaches have been developed to preserve privacy in distributed computation and optimization. Cryptographic methods, including homomorphic encryption~\cite{zhang2018enabling} and secure multiparty computation~\cite{li2026randomized}, enable computations on protected data, at the expense of additional cryptographic operations and communication. Another line of work uses structured perturbations to conceal private information while preserving the accuracy~\cite{gao2022algorithm, maithripala2026privacy, gade2018private}. However, privacy guarantees based on structured perturbations often rely on restrictions on the adversary’s access to information. In particular, protecting an agent may require at least one of its neighbors to remain non-colluding, a condition that can be difficult to ensure in general multi-agent systems. Differential privacy (DP) offers a widely adopted framework for quantifying privacy protection through independent additive noises~\cite{chen2023differentially,xuan2023gradienttracking,ding2022state}. The noise must be carefully calibrated to control the cumulative privacy loss while maintaining optimization accuracy.

Incorporating DP into gradient tracking with local updates offers a promising approach to improving communication efficiency while protecting private information. An additional local update allows further computation without injecting fresh noise, potentially improving optimization performance per communication round under the same privacy budget. However, this update changes both the convergence dynamics and the sensitivity of communicated information, requiring a separate convergence and privacy analysis. To address this challenge, we propose in this paper a differentially private gradient-tracking algorithm with one local update between consecutive communication rounds. The algorithm perturbs both the primal state and the tracking direction only at communication iterations. To the best of our knowledge, this is the first work to establish both convergence and infinite-horizon pure differential privacy for this alternating local-update gradient-tracking scheme under the stated objective adjacency relation.

Our main contributions are summarized as follows.
\begin{enumerate}
\item We show that communication-free local updates alone do not protect gradient information. Specially, an attacker who knows the mixing weights and stepsize and observes the primal and tracking states at each communication iteration can uniquely reconstruct all uncommunicated intermediate states and every agent’s complete local gradient trajectory.

\item We develop a differentially private gradient-tracking algorithm that perturbs both the primal state and the gradient-tracking direction only during communication, while performing local updates without fresh noise injection. We establish sufficient conditions for convergence to a common random limiting point and characterize this point through the accumulated communication perturbations.

\item We derive a matched-transcript sensitivity recursion and an explicit infinite-horizon pure differential-privacy budget under the stated affine adjacency. Joint conditions on the stepsize and geometric noise-decay rates ensure convergence and privacy. 
\item Numerical comparisons demonstrate improved convergence per communication round relative to an existing differentially private method under equal certified privacy budgets in the tested setting.
\end{enumerate}

\section{Problem Formulation and Preliminaries}\label{sec:problem}
Consider a network of $n$ agents communicating over a connected undirected graph $\mathcal{G}=(\mathcal{V},\mathcal{E})$, where $\mathcal{V}=\{1,2, \ldots,n\}$. Agent $i$ knows only its local objective function $f_i:\mathbb{R}^d\to\mathbb{R}$. The agents cooperatively solve
\begin{equation}\label{eq:problem}
    \min_{x\in\mathbb{R}^d} \;\; f(x):=\frac{1}{n}\sum_{i=1}^n f_i(x).
\end{equation}
Each agent maintains a local copy $x_i\in\mathbb{R}^d$.

\begin{assumption}\label{ass:function}
For each agent $i$, the local objective $f_i:\mathbb{R}^d\rightarrow\mathbb{R}$ is $\mu$-strongly convex and its gradient is $L$-Lipschitz continuous. In particular, for any $x,x'\in\mathbb{R}^d$,
$
\bigl\langle
\nabla f_i(x)-\nabla f_i(x'),
\,x-x'
\bigr\rangle
\ge
\mu \|x-x'\|^2,
$
and
$
\|\nabla f_i(x)-\nabla f_i(x')\|
\le
L\|x-x'\|.
$
\end{assumption}

\begin{assumption}\label{ass:mixing}
The mixing matrix $W=[w_{ij}]\in\mathbb{R}^{n\times n}$ is symmetric and doubly stochastic, i.e.,
$
W\mathbf{1}=\mathbf{1},
\;
\mathbf{1}^{\mathrm T}W=\mathbf{1}^{\mathrm T}.
$
The graph induced by $W$ is connected.

Let $J:=\mathbf1\mathbf1^{\rm T}/n$, $\lambda:=\|W-J\|_2<1$, and $\sigma:=\|W-I\|_2$. 
\end{assumption}

\subsection{Differential Privacy}

We first define adjacency between problem instances and
then state the differential-privacy requirement.

\begin{definition}\label{def:privacy_adjacency}
Two problem instances $\mathcal P$ and $\mathcal P'$ are adjacent at agent $i$ if they have the same graph, algorithmic parameters, and initial primal states, and
\[
f_j'=f_j \quad (j\ne i),
\qquad
f_i'(x)=f_i(x)+a_i^{\rm T}x+b_i,
\] 
where $a_i\in\mathbb R^d$, $b_i\in\mathbb R$, and $\|a_i\|_2\le\delta_i$ for a prescribed $\delta_i>0$. 
\end{definition}

\begin{definition}\label{def:dp}
Let $\mathcal M(\mathcal P)$ denote the complete communication transcript generated by a randomized distributed algorithm on problem instance $\mathcal P$. The algorithm is $\epsilon$-differentially private if, for every pair of adjacent instances $\mathcal P,\mathcal P'$ and every measurable set $\mathcal O$ of communication transcripts,
\[
\Pr[\mathcal M(\mathcal P)\in\mathcal O]
\le e^\epsilon
\Pr[\mathcal M(\mathcal P')\in\mathcal O],
\]
where the probabilities are taken over the algorithm's random perturbations. $\epsilon$-differential privacy ensures that an adversary can hardly identify a local objective function among all possible ones, even when all the observation sequences are disclosed. A smaller $\epsilon$ implies a higher privacy level.
\end{definition}

\section{Privacy Vulnerability of Local-Update Gradient Tracking}

Local updates have recently been incorporated into DIGing to reduce the communication required for distributed optimization. Reference \cite{wang2026localupdates} showed that, with an appropriately selected stepsize, the maximal improvement is achieved when there is exactly one communication-free local update between two consecutive communication rounds. 

Accordingly, we consider one communication-free local update between two consecutive communication rounds. Define the stacked vectors
$x^k := [(x_1^k)^{\rm T}\; (x_2^k)^{\rm T}\ldots(x_n^k)^{\rm T}]^{\rm T}$,
$y^k := [(y_1^k)^{\rm T}\; (y_2^k)^{\rm T}\ldots(y_n^k)^{\rm T}]^{\rm T}$,
and
$g^k :=
[(\nabla f_1(x_1^k))^{\rm T}$ $(\nabla f_2(x_2^k))^{\rm T}\ldots
(\nabla f_n(x_n^k))^{\rm T}]^{\rm T}$,
all belonging to $\mathbb R^{nd}$.
When acting on these stacked vectors, $W$ denotes $W\otimes I_d$.

With $x_i^0=0$ and $y_i^0=g_i^0=\nabla f_i(x_i^0)$,
the deterministic gradient-tracking recursion consists of
communication updates
\begin{equation}\label{eq:communication_update}
\begin{cases}
x^{k+1}=Wx^k-\alpha y^k,\\
y^{k+1}=Wy^k+g^{k+1}-g^k,
\end{cases}
\quad k\in\{1,3,5,\ldots\},
\end{equation}
and local updates
\begin{equation}\label{eq:local_update}
\begin{cases}
x^{k+1}=x^k-\alpha y^k,\\
y^{k+1}=y^k+g^{k+1}-g^k,
\end{cases}
\quad k\in\{0,2,4,\ldots\}.
\end{equation}

The following result shows that the communicated information can reveal the complete local gradient trajectories.

\begin{theorem}
\label{thm:gradient_leakage}
Consider the local update gradient tracking algorithm~\eqref{eq:communication_update}--\eqref{eq:local_update} with
$
x_i^0=0\in \mathbb{R}^d,\;y_i^0=g_i^0,\; \alpha>0.
$
Assume that an attacker knows $W$ and $\alpha$ and observes $x^k$ and $y^k$ at every communication iteration $k\in\{1,3,5,\ldots\}$. Then, the complete sequences $\{x^k\}_{k\ge0}$, $\{y^k\}_{k\ge0}$, and $\{g^k\}_{k\ge0}$ can be uniquely reconstructed. Consequently, the local gradient trajectory $\{\nabla f_i(x_i^k)\}_{k\ge0}$ is revealed for every agent $i$.
\end{theorem}

\begin{proof}
At the first local iteration,
$
x^1=x^0-\alpha y^0,\;
y^1=y^0+g^1-g^0.
$
Since $x^0=0$ is known and $y^0=g^0$, the attacker's observations $x^1$ and $y^1$ give
$
g^0=y^0=\frac{x^0-x^1}{\alpha},
\;
g^1=y^1.
$

We now proceed recursively. Let $k$ be an arbitrary odd communication iteration and suppose that $x^k$, $y^k$, and $g^k$ are known. From the communication update,
\[
x^{k+1}=Wx^k-\alpha y^k,
\quad
y^{k+1}=Wy^k+g^{k+1}-g^k,
\]
the state $x^{k+1}$ can be computed directly. The next iteration is local, and therefore
\[
x^{k+2}=x^{k+1}-\alpha y^{k+1},
\quad
y^{k+2}=y^{k+1}+g^{k+2}-g^{k+1}.
\]
Since $x^{k+2}$ and $y^{k+2}$ are observed at the next communication iteration, the intermediate tracking state and the two successive gradients are obtained as
\[
\begin{aligned}
y^{k+1}
&=\frac{x^{k+1}-x^{k+2}}{\alpha},\\
g^{k+1}
&=y^{k+1}-Wy^k+g^k,\\
g^{k+2}
&=y^{k+2}-y^{k+1}+g^{k+1}.
\end{aligned}
\]

Hence, knowledge of $(x^k,y^k,g^k)$ at one communication iteration, together with the
observation $(x^{k+1}, y^{k+1})$ , uniquely determines
$
x^{k+1},\;y^{k+1},\; g^{k+1}, \text{and} \; g^{k+2}.
$
Starting from the reconstructed values $g^0$ and $g^1$, this argument can be applied successively for $k=1,3,5,\ldots$. Therefore,
$
x^k,\; y^k, \text{and}\; g^k
$
are uniquely determined for all $k\ge0$. Since the $i$th $d$-dimensional block of $g^k$ is $g_i^k=\nabla f_i(x_i^k)$, the complete local gradient trajectory $\{\nabla f_i(x_i^k)\}_{k\ge0}$ is revealed for every agent $i$.
\end{proof}

\section{Communication-Efficient Differentially Private Gradient Tracking Algorithm}

The analysis in the previous section shows that reducing the communication frequency alone does not protect the private information contained in gradient tracking. In particular, although one communication-free local update is performed between two consecutive communication rounds, the states exchanged at the communication iterations are sufficient for an attacker to reconstruct the complete gradient trajectory. Therefore, an additional privacy mechanism is required to prevent the communicated information from revealing the
underlying states and local gradients.

Motivated by this observation, we introduce random perturbations into the information exchanged at communication iterations. The proposed method has two main objectives. First, the state information communicated by each agent is perturbed to prevent an attacker from directly observing its intermediate optimization state. Second, the gradient-tracking variable is independently perturbed to prevent the exact reconstruction of local gradient trajectories.

\subsection{Privacy-Preserving Gradient-Tracking Algorithm}
Let
$
\mathcal{C}:=\{1,3,5,\ldots\}
$
denote the set of communication iterations. For every
$k\in\mathcal{C}$, agent $i$ generates two independent random perturbations,
$\xi_i^k$ and $\nu_i^k$, and transmits
$
\widetilde{x}_i^k=x_i^k+\xi_i^k \;\text{and}
\;
\widetilde{y}_i^k=y_i^k+\nu_i^k
$
instead of its actual state $x_i^k$ and tracking variable $y_i^k$. 

The perturbations are chosen independently across agents,
communication updates, and the primal and tracking variables
$x_i^k$ and $y_i^k$. Specifically,
$
\xi_i^k
\sim \operatorname{Lap}_d\!\left(d_{x,i}q_i^{k-1}\right),
\;
\nu_i^k
\sim \operatorname{Lap}_d\!\left(d_{y,i}q_i^{k-1}\right),
$
where
$
d_{x,i}>0,
\;
d_{y,i}>0,
\;
0<q_i<1.
$
Here, $\operatorname{Lap}_d(b)$ denotes the distribution of
a $d$-dimensional vector with independent zero-mean Laplace
coordinates of scale $b$. The parameters $d_{x,i}$ and $d_{y,i}$ determine the initial perturbation magnitudes of the state and tracking variables, respectively, while $q_i$ determines its decay rate. 
An important feature of the proposed method is that privacy perturbations are introduced only when information is exchanged. At every even iteration, the agents perform the original communication-free local update without adding random perturbations. The resulting algorithm is summarized in Algorithm~\ref{alg:private}.
\begin{algorithm}[t]
\caption{Privacy-Preserving Gradient Tracking with Local Updates}
\label{alg:private}
\begin{algorithmic}[1]
\REQUIRE Initial point $x_i^0=0$, stepsize $\alpha>0$, $y_i^0=\nabla f_i(x_i^0)$,
and total number of iterations $K$.
\FOR{$k=0,1,\ldots,K-1$}
    \IF{$k$ is odd}
        \STATE \textbf{Communication update:}
        \STATE Each agent $i \in \mathcal{V}$ draws $\xi_i^k$ and $\nu_i^k$ according to
        the prescribed noise distributions.
        \STATE $\widetilde x_i^k=x_i^k+\xi_i^k$, \quad
               $\widetilde y_i^k=y_i^k+\nu_i^k$.
        \STATE Each agent $i \in \mathcal{V}$ broadcasts $\widetilde x_i^k,\widetilde y_i^k$
        and receives neighbors' messages.
        \STATE Each agent $i \in \mathcal{V}$ updates $\displaystyle
        x_i^{k+1}
        =\sum_{j\in\mathcal N_i}
        w_{ij}\widetilde x_j^k-\alpha y_i^k$ and $\displaystyle
        y_i^{k+1}
        =\sum_{j\in\mathcal N_i}
        w_{ij}\widetilde y_j^k
        +\nabla f_i(x_i^{k+1})-\nabla f_i(x_i^k)$.
    \ELSE
        \STATE \textbf{Local update:}
        \STATE Each agent $i \in \mathcal{V}$ updates $x_i^{k+1}=x_i^k-\alpha y_i^k$ and $y_i^{k+1}=y_i^k
        +\nabla f_i(x_i^{k+1})-\nabla f_i(x_i^k)$.
    \ENDIF
\ENDFOR
\end{algorithmic}
\end{algorithm}

The introduction of the direction perturbation modifies the exact gradient-tracking identity of the deterministic algorithm. We therefore first characterize this modified tracking relation and the resulting limiting point in the next section. The convergence and differential-privacy properties of the proposed method are then established subsequently.

\subsection{Aggregate tracking relation and candidate limiting point}
Since $y_i^0=\nabla f_i(x_i^0)$, we have
$\sum_i y_i^0=\sum_i\nabla f_i(x_i^0)$.
For a local iteration $k\notin\mathcal C$,
\[
\sum_{i=1}^{n}y_i^{k+1}
=
\sum_{i=1}^{n}y_i^k
+
\sum_{i=1}^{n}
\bigl(
\nabla f_i(x_i^{k+1})-\nabla f_i(x_i^k)
\bigr).
\]
For a communication iteration $k\in\mathcal C$, double stochasticity of $W$ gives
\[
\sum_{i=1}^{n}y_i^{k+1}
=
\sum_{i=1}^{n}y_i^k
+
\sum_{i=1}^{n}\nu_i^k
+
\sum_{i=1}^{n}
\bigl(
\nabla f_i(x_i^{k+1})-\nabla f_i(x_i^k)
\bigr).
\]
Summing these updates from iteration $0$ to $k-1$,
the intermediate gradient terms cancel each other out. Using the initialization yields
\begin{equation}\label{eq:aggregate_tracking}
\sum_{i=1}^{n}y_i^k
=
\sum_{i=1}^{n}\nabla f_i(x_i^k)
+
\sum_{\substack{\ell\in\mathcal C\\\ell<k}}
\sum_{i=1}^{n}\nu_i^\ell.
\end{equation}
Thus, communication-free local updates introduce no additional perturbations. The perturbations $\nu_i^k$ added to the tracking variables $y_i^k$ accumulate only at communication iterations $\ell\in\mathcal C$, whereas the perturbations $\xi_i^k$ added to the primal variables $x_i^k$ affect the primal trajectory without contributing a separate term to~\eqref{eq:aggregate_tracking}.

At each $k\in\mathcal C$, the noise scale of $\nu_i^k$ is
$d_{y,i}q_i^{k-1}$, where $0<q_i<1$.
Since $\mathbb E|Z|=b$ for $Z\sim\operatorname{Lap}(b)$,
\[
\sum_{k\in\mathcal C}\mathbb E\|\nu_i^k\|_1
=\frac{d\,d_{y,i}}{1-q_i^2}<\infty.
\]
The series of perturbations added to the tracking variables
therefore converges absolutely almost surely. Since the number of agents is finite,
the aggregate perturbation
$
\Delta_\nu:=
\sum_{\ell\in\mathcal C}\sum_{i=1}^{n}\nu_i^\ell
$
is well defined almost surely.

Define the candidate limiting point 
\[x^\infty \!\!\!=\! \operatorname{argmin}\{ \sum_i f_i(x)+\langle\Delta_\nu,x\rangle\},\] 
which satisfies \begin{equation}\label{eq:limiting_optimality}
\sum_{i=1}^{n}\nabla f_i(x^\infty)=-\Delta_\nu.
\end{equation} 
Strong convexity ensures the existence and uniqueness of $x^\infty$ almost surely. We next establish convergence to this point under sufficient stepsize conditions.

\section{Convergence Analysis}
\label{sec:convergence}

We analyze convergence by combining the error bounds for the local and communication updates. When acting on stacked vectors in $\mathbb R^{nd}$, $W$ and $J$ represents $W\otimes I_d$ and $J\otimes I_d$, respectively, where $J:=\mathbf1\mathbf1^{\rm T}/n$. Define $q:=\max_{i\in\mathcal V}q_i<1$.

\subsection{Error Variables and Perturbation Tails}

Define the averages $\bar x^k:=n^{-1}\sum_i x_i^k\in\mathbb R^d$ and $\bar y^k:=n^{-1}\sum_i y_i^k\in\mathbb R^d$, and the disagreements
\[
\check x^k:=x^k-\mathbf1_n\otimes\bar x^k\in\mathbb R^{nd},
\qquad
\check y^k:=y^k-\mathbf1_n\otimes\bar y^k\in\mathbb R^{nd},
\]
where $\mathbf1_n\in\mathbb R^n$ is the all-ones vector.
Define
\begin{equation}\label{eq:direction_tail}
r^k:=\frac1n
\sum_{\substack{\ell\in\mathcal C\\\ell\ge k}}
\sum_{i=1}^{n}\nu_i^\ell.
\end{equation}

\begin{lemma}\label{lem:tracker_tail}
Suppose Assumptions~\ref{ass:function} and~\ref{ass:mixing} and \eqref{eq:limiting_optimality} hold. Let $q:=\max_{i\in\mathcal V}q_i<1$.
Then, for every $k\ge0$,
\begin{equation}\label{eq:average_tracker_bound}
\|\bar y^k\|_2
\le L\|\bar x^k-x^\infty\|_2
+\frac L{\sqrt n}\|\check x^k\|_2+\|r^k\|_2.
\end{equation}
Moreover,
\begin{equation}\label{eq:tail_mean_bound}
\mathbb E\|r^k\|_2
\le c_rq^{2\lfloor k/2\rfloor},
\end{equation}
where
$c_r:=d\sum_{i=1}^n d_{y,i}/[n(1-q^2)]<\infty$.
\end{lemma}

\begin{proof}
By the definitions of $\Delta_\nu$ and $r^k$, we have
$
\sum_{\substack{\ell\in\mathcal C\\ \ell<k}}
\sum_{i=1}^{n}\nu_i^\ell
=\Delta_\nu-nr^k.
$
Substituting this sum into \eqref{eq:aggregate_tracking} and applying \eqref{eq:limiting_optimality} gives
$
\bar y^k
=\frac1n\sum_{i=1}^{n}
\bigl[\nabla f_i(x_i^k)-\nabla f_i(x^\infty)\bigr]-r^k.
$

Adding and subtracting $\nabla f_i(\bar x^k)$ and using $L$-Lipschitz continuity yields
$
\|\bar y^k\|_2
\le L\|\bar x^k-x^\infty\|_2
+\frac Ln\sum_{i=1}^n\|x_i^k-\bar x^k\|_2+\|r^k\|_2.
$
By Cauchy--Schwarz inequality,
$\sum_{i=1}^n\|x_i^k-\bar x^k\|_2
\le\sqrt n\,\|\check x^k\|_2$,
which leads to \eqref{eq:average_tracker_bound}.

Next, set $c_\nu:=d\sum_i d_{y,i}$. Since $\mathbb E|Z|=b$ for $Z\sim\operatorname{Lap}(b)$,
$
\sum_{i=1}^n\mathbb E\|\nu_i^\ell\|_1
=d\sum_{i=1}^n d_{y,i}q_i^{\ell-1}
\le c_\nu q^{\ell-1},\ell\in\mathcal C.
$
The smallest communication index $\ell\ge k$ in (6) is
$2\lfloor k/2\rfloor+1$, where $\lfloor\cdot\rfloor$
denotes the floor function. Using the triangle inequality, the fact that $\|v\|_2\le\|v\|_1$, and the geometric-series formula gives
\[
\mathbb E\|r^k\|_2
\le \frac1n
\sum_{\substack{\ell\in\mathcal C\\\ell\ge k}}
\sum_{i=1}^n\mathbb E\|\nu_i^\ell\|_1
\le \frac{c_\nu q^{2\lfloor k/2\rfloor}}{n(1-q^2)}
=c_rq^{2\lfloor k/2\rfloor},
\]
which proves \eqref{eq:tail_mean_bound}. 
\end{proof}

The stacked perturbations satisfy
$
\mathbb E\|\xi^k\|_2\le c_\xi q^{k-1},
\;
\mathbb E\|\nu^k\|_2\le c_\nu q^{k-1},
\; k\in\mathcal C,
$
where $c_\xi:=d\sum_i d_{x,i}$.
These bounds will be used in the two-step comparison recursion.

\subsection{Two-Step Comparison Recursion}

We track three errors: the distance of the network average from $x^\infty$, the disagreement among local estimates, and the disagreement among gradient trackers. Define
\[
z^k:=
\begin{bmatrix}
\|\bar x^k-x^\infty\|_2\\
\|\check x^k\|_2\\
\|\check y^k\|_2
\end{bmatrix},
\quad
G(u):=\frac1n\sum_{i=1}^n\nabla f_i(u).
\]

\begin{lemma}\label{lem:2}
Suppose Assumptions~\ref{ass:function} and~\ref{ass:mixing} hold. Let $0<\alpha<2\mu/L^2$ and $\kappa_\alpha:=\sqrt{1-2\alpha\mu+\alpha^2L^2}<1$. For every even $k\ge0$,
\begin{equation}\label{eq:two_step_comparison}
z^{k+2}\le R(\alpha)z^k+m^k,
\end{equation} 
where $R(\alpha):=A_{\rm C}A_{\rm L}$ is componentwise nonnegative and $m^k:=A_{\rm C}u_{\rm L}^k+u_{\rm C}^{k+1}$. Here and throughout, vector inequalities are interpreted componentwise, and a matrix is said to be nonnegative (respectively, positive) if all its entries are nonnegative (respectively, positive).
\end{lemma}

\begin{proof}
We first bound these errors during each update, then combine one local update with the following communication update.

The average gradient $G$ is $\mu$-strongly monotone and $L$-Lipschitz. Expanding the squared norm therefore gives
\begin{align*}
&\quad \;\|u-v-\alpha[G(u)-G(v)]\|_2^2\\
&=\|u-v\|_2^2
-2\alpha\langle u\!-\!v,G(u)\!-\!G(v)\rangle
\!+\!\alpha^2\|G(u)-G(v)\|_2^2\\
&\le(1-2\alpha\mu+\alpha^2L^2)\|u-v\|_2^2.
\end{align*}
Hence, we obtain
\begin{equation}\label{eq:shifted_gradient_contraction}
\|u-v-\alpha[G(u)-G(v)]\|_2
\le\kappa_\alpha\|u-v\|_2.
\end{equation}

Averaging the local updates of all agents gives $\bar x^{k+1}=\bar x^k-\alpha\bar y^k$. Define the gradient mismatch $e^k:=n^{-1}\sum_i[\nabla f_i(x_i^k)-\nabla f_i(\bar x^k)]$. Lipschitz continuity and Cauchy--Schwarz inequality give
\begin{equation}
\|e^k\|_2
\le\frac Ln\sum_i\|x_i^k-\bar x^k\|_2
\le\frac L{\sqrt n}\|\check x^k\|_2.
\label{eq:e_k}
\end{equation}

Subtracting $x^\infty$ from the average update
$\bar x^{k+1}=\bar x^k-\alpha\bar y^k$ and substituting
the tracking identity
$\bar y^k=G(\bar x^k)-G(x^\infty)+e^k-r^k$, we obtain
\[
\begin{aligned}
\bar x^{k+1}-x^\infty
={}&\bar x^k-x^\infty
-\alpha\bigl[G(\bar x^k)-G(x^\infty)\bigr]\\
&-\alpha e^k+\alpha r^k.
\end{aligned}
\]
Applying the triangle inequality, the contraction bound \eqref{eq:shifted_gradient_contraction}
with $u=\bar x^k$ and $v=x^\infty$, and \eqref{eq:e_k} gives
\begin{equation}
\begin{aligned}
\|\bar x^{k+1}-x^\infty\|_2
\le{}&\kappa_\alpha\|\bar x^k-x^\infty\|_2\\
&+\frac{\alpha L}{\sqrt n}\|\check x^k\|_2
+\alpha\|r^k\|_2.
\end{aligned}
\label{eq:local_average_bound}
\end{equation}

Subtracting the average update
$\bar{x}^{k+1}=\bar{x}^k-\alpha\bar{y}^k$
from each agent's local update
$x_i^{k+1}=x_i^k-\alpha y_i^k$ gives $\check x^{k+1}=\check x^k-\alpha\check y^k$, and hence
\begin{equation}\label{eq:local_disagreement_bound}
\|\check x^{k+1}\|_2
\le\|\check x^k\|_2+\alpha\|\check y^k\|_2.
\end{equation}

For the tracker update, subtracting the average is an orthogonal projection and cannot increase the norm. Lipschitz continuity and $x^{k+1}-x^k=-\alpha y^k$ therefore imply that
\[
\|\check y^{k+1}\|_2
\le\|\check y^k\|_2+L\|x^{k+1}-x^k\|_2
=\|\check y^k\|_2+\alpha L\|y^k\|_2.
\]
We use $\|y^k\|_2\le\|\check y^k\|_2+\sqrt n\|\bar y^k\|_2$ and \eqref{eq:average_tracker_bound} to arrive at
\begin{align}
\|\check y^{k+1}\|_2
\le{}&\alpha\sqrt nL^2\|\bar x^k-x^\infty\|_2
+\alpha L^2\|\check x^k\|_2\nonumber\\
&+(1+\alpha L)\|\check y^k\|_2
+\alpha\sqrt nL\|r^k\|_2.
\label{eq:local_tracker_bound}
\end{align}
Combining the three bounds \eqref{eq:local_average_bound} \eqref{eq:local_disagreement_bound} and \eqref{eq:local_tracker_bound} yields \begin{equation}\label{eq:local_comparison}
    z^{k+1}\le A_{\rm L}z^k+u_{\rm L}^k, \quad k \in \{0,2,4,\ldots, \}
    \end{equation}
where
\[
\begin{aligned}
A_{\rm L}\!\!
&:=\!\!\!
\begin{bmatrix}
\kappa_\alpha \!& \alpha\! L/\sqrt n & 0\\
0 & 1 & \alpha\\
\alpha\sqrt nL^2 & \alpha L^2 & 1+\alpha L
\end{bmatrix},
u_{\rm L}^k\!\!\!\!\!
&:=
\alpha\|r^k\|_2 \!\!
\begin{bmatrix}
1\\0\\\sqrt nL
\end{bmatrix}.
\end{aligned}
\]

Averaging the communication update gives $\bar x^{k+1}=\bar x^k-\alpha\bar y^k+\bar\xi^k$. Repeating the average-error argument and subtracting the average from the stacked update yield
\begin{align}
\|\bar x^{k+1}-x^\infty\|_2
\le{}&\kappa_\alpha\|\bar x^k-x^\infty\|_2
+\frac{\alpha L}{\sqrt n}\|\check x^k\|_2\nonumber\\
&+\alpha\|r^k\|_2+\|\bar\xi^k\|_2.
\label{eq:x_disagrrement_bound}
\end{align}

Subtracting the network average from the communication
update and using $W\mathbf1_n=\mathbf1_n$ gives
\[
\begin{aligned}
\check x^{k+1}
&=x^{k+1}-\mathbf1_n\otimes\bar x^{k+1}\\
&=W\check x^k-\alpha\check y^k+W\check\xi^k\\
&=(W-J)\check x^k-\alpha\check y^k
  +(W-J)\check\xi^k,
\end{aligned}
\]
where the last equality follows from
$J\check x^k=J\check\xi^k=0$.
Applying the triangle inequality and
$\|W-J\|_2=\lambda$, we obtain
\begin{align}
\|\check x^{k+1}\|_2
\le{}&\lambda\|\check x^k\|_2
+\alpha\|\check y^k\|_2
+\lambda\|\check\xi^k\|_2.
\label{eq:communication_disagreement_bound}
\end{align}
Similarly, the tracker disagreement satisfies
\begin{equation}
\|\check y^{k+1}\|_2
\le\lambda\|\check y^k\|_2
+\lambda\|\check\nu^k\|_2
+L\|x^{k+1}-x^k\|_2.
\label{eq:yerror_disagreement}
\end{equation}
To bound the last term in~\eqref{eq:yerror_disagreement}, we use $x^{k+1}-x^k=(W-I)\check x^k-\alpha y^k+W\xi^k$, $\|W\|_2=1$, and $\|y^k\|_2\le\|\check y^k\|_2+\sqrt n\|\bar y^k\|_2$, and obtain
\begin{align*}
\|x^{k+1}-x^k\|_2
\le{}&\alpha\sqrt nL\|\bar x^k-x^\infty\|
+(\sigma+\alpha L)\|\check x^k\|_2\\
&+\alpha\|\check y^k\|_2
+\alpha\sqrt n\|r^k\|_2+\|\xi^k\|_2.
\end{align*}
Combining \eqref{eq:x_disagrrement_bound}, ~\eqref{eq:communication_disagreement_bound}, and \eqref{eq:yerror_disagreement}  gives
\begin{equation}\label{eq:communication_comparison}
z^{k+1}\le A_{\rm C}z^k+u_{\rm C}^k, \quad k \in \mathcal{C},
\end{equation}
where
\[
A_{\rm C}:=
\begin{bmatrix}
\kappa_\alpha&\alpha L/\sqrt n&0\\
0&\lambda&\alpha\\
\alpha\sqrt nL^2&L(\sigma+\alpha L)&\lambda+\alpha L
\end{bmatrix},
\]
\[
u_{\rm C}^k:=
\begin{bmatrix}
\alpha\|r^k\|_2+\|\bar\xi^k\|_2\\
\lambda\|\check\xi^k\|_2\\
\alpha\sqrt nL\|r^k\|_2+L\|\xi^k\|_2
+\lambda\|\check\nu^k\|_2
\end{bmatrix}.
\]

For an even $k$, the local step is followed by communication. Since $A_{\rm C}$ has nonnegative entries, substituting
\eqref{eq:local_comparison} into \eqref{eq:communication_comparison} gives \eqref{eq:two_step_comparison}
\end{proof}

By Lemma \ref{lem:2}, $R(\alpha)$ propagates the existing errors, while $m^k$ collects the perturbation contributions over both steps. The perturbation-tail bounds imply that 
$
\mathbb E[m^k]\le c_mq^k\mathbf1,
\; k\in\{0,2,4,\ldots\}
$ for some $c_m>0$.
Therefore, the expected perturbation contribution is bounded by a geometrically decaying sequence, with a factor $q^2$ per pair of local and communication updates. We next show that stability of matrix $R(\alpha)$ makes the errors decay as well.

\subsection{Convergence in Mean}

\begin{theorem}\label{thm:mean_convergence} Suppose Assumptions~\ref{ass:function} and~\ref{ass:mixing} hold, and the perturbations are generated as specified in Algorithm~\ref{alg:private}. Let
\begin{equation}\label{eq:convergence_stepsize}
0<\alpha<\frac{2\mu}{L^2},
\qquad \rho(R(\alpha))<1.
\end{equation}
For every $p\in(\max\{\rho(R(\alpha)),q^2\},1)$, there is a constant $C_p<\infty$ such that
\begin{equation}\label{eq:mean_linear_rate}
\mathbb E\|x_i^k-x^\infty\|_2+\mathbb E\|y_i^k\|_2
\le C_pp^{\lfloor k/2\rfloor},
\quad i\in\mathcal V,\quad k\ge0.
\end{equation}
Consequently, $x_i^k\to x^\infty$ and $y_i^k\to0$ in mean.
\end{theorem}

\begin{proof}
The initial errors have finite expectation. Indeed, strong monotonicity gives
$\|x^\infty-x^\star\|_2\le\|\Delta_\nu\|_2/(n\mu)$, where $\mathbb E\|\Delta_\nu\| <\infty$. Together with initialization, this implies that $\mathbb E[z^0]$ is finite.

Write $R:=R(\alpha)$. Taking expectations in \eqref{eq:two_step_comparison} and repeatedly substituting the preceding bound gives, for even $k\ge2$,
\begin{align}
\mathbb E[z^k]
\le{}&R^{k/2}\mathbb E[z^0]+c_m\sum_{j=0}^{k/2-1}
R^{k/2-1-j}q^{2j}\mathbf1.
\label{eq:expected_error_expansion}
\end{align}
The first term is the effect of the initial errors; the sum accounts for perturbations from previous update pairs. Choose $\widehat p$ with $\max\{\rho(R),q^2\}<\widehat p<p$. By the spectral-radius formula, there exists a finite $c_R>0$ such that $\|R^j\|_2\le c_R\widehat p^{\,j}$ for all $j\ge0$. For even $k\ge2$, the sum is bounded as
\[
\sum_{j=0}^{k/2-1}\widehat p^{\,k/2-1-j}q^{2j}
=
\frac{\widehat p^{\,k/2}-q^k}{\widehat p-q^2}
\le
\frac{\widehat p^{\,k/2}}{\widehat p-q^2}.
\]
Combining this estimate with the bound on $R^{k/2}$ and using $\widehat p<p$ yields $\|\mathbb E[z^k]\|_2\le c_p p^{k/2}$ for any even $k$, including $k=0$, with a suitable finite constant $c_p$. For any odd $k$, the local update error propagation bound~\eqref{eq:local_comparison} applied at $k- 1$ and the residual tail bound
$\mathbb E\|r^{k-1}\|_2
\le c_r q^{k-1}\le c_r p^{(k-1)/2}$
give, after enlarging $c_p$ if necessary,
$
\|\mathbb E[z^k]\|_2
\le c_p p^{\lfloor k/2\rfloor},
\, k\ge0.
$

To bound the errors of each agent, we use $\|x_i^k-x^\infty\|_2
\le\|\bar x^k-x^\infty\|_2+\|\check x^k\|_2$
and
$\|y_i^k\|_2\le\|\check y^k\|_2+\|\bar y^k\|_2$. Combining these with \eqref{eq:average_tracker_bound} and \eqref{eq:tail_mean_bound} results in \eqref{eq:mean_linear_rate}. 
\end{proof}

The stepsize condition holds for sufficiently small positive $\alpha$. Indeed,
\[
R(0)=
\begin{bmatrix}
1&0&0\\
0&\lambda&0\\
0&L\sigma&\lambda
\end{bmatrix},
\]
with eigenvalues $1,\lambda,\lambda$. Since $0\le\lambda<1$, only the eigenvalue at $1$ needs to move inside the unit circle.

Let $\vartheta(\alpha)$ denote the eigenvalue of $R(\alpha)$ satisfying $\vartheta(0)=1$. This eigenvalue is simple at $\alpha=0$, with left and right eigenvectors $e_1=[1,0,0]^{\rm T}$. Since $R_{11}(\alpha)=\kappa_\alpha^2 =1 2\mu\alpha+L^2\alpha^2$, the derivative formula for a simple eigenvalue gives
$
\vartheta'(0)=e_1^{\rm T}\frac{d}{dt}R(0)e_1=-2\mu<0.
$ Consequently, $\vartheta(\alpha)=1-2\mu\alpha+O(\alpha^2)$, so this eigenvalue lies strictly between $0$ and $1$ for sufficiently small positive $\alpha$. By continuity, the other two eigenvalues remain near $\lambda$ and also have magnitude less than $1$. Thus, $\rho(R(\alpha))<1$ for sufficiently small positive $\alpha$. Choosing $\alpha$ small enough also ensures $\alpha<2\mu/L^2$. Therefore, (21) provides computable sufficient conditions for convergence.

Thus, all agents converge in mean to the same random perturbed optimizer. If any direction-noise scale is positive, $\mathbb E\|x^\infty-x^\star\|_2^2>0$. State perturbations affect the transient behavior but do not enter the limiting optimality condition. Setting all direction perturbations to zero recovers $x^\infty=x^\star$ under the convergence conditions.

\section{Differential-Privacy Analysis}
\label{sec:privacy}

We analyze privacy with respect to the complete communication transcript
$
\mathcal M(\mathcal P)
:=\{(\widetilde x^k,\widetilde y^k):k\in\mathcal C\}.
$
The observer knows the graph, weights, stepsize, noise distributions, and initial primal states, but does not observe the agents' internal states or realized perturbations. All algorithmic parameters, including the noise scales, are fixed independently of the private objective within the admissible problem class. 

\subsection{Adjacency and Protected Information}

We consider adjacent problem instances $\mathcal P$ and $\mathcal P'$ as defined in Definition~\ref{def:privacy_adjacency}. The initial primal states coincide, while the tracking variables are initialized separately as
$
y_j^0=\nabla f_j(x_j^0),
\quad
y_j^{\prime0}=\nabla f_j'(x_j^0).
$
Consequently, $y_i^{\prime0}-y_i^0=a_i$ at the changed agent, whereas $y_j^{\prime0}=y_j^0$ for $j\ne i$. The affine difference between adjacent objectives implies
the identical gradient increment property
\begin{equation}\label{eq:identical_gradient_increments}
\nabla f_i'(u)-\nabla f_i'(v)
=
\nabla f_i(u)-\nabla f_i(v),
\quad u,v\in\mathbb R^d.
\end{equation}
This property will be used in the subsequent analysis to bound the differences between internal states under adjacent instances.

\subsection{Sensitivity Under a Common Transcript}

Fix two adjacent instances that differ at agent $i$, and consider the same finite sequence of transmitted messages under both instances. For every $j\ne i$, the objective, initial states, and received messages are identical, so the internal states coincide throughout this sequence. We therefore focus on agent $i$ and define
$
\Delta x^k:=x_i^k-x_i^{\prime k},
\,
\Delta y^k:=y_i^k-y_i^{\prime k},
$
and
$h^k:=\nabla f_i(x_i^k)-\nabla f_i(x_i^{\prime k})$.
Lipschitz continuity results in
$\|h^k\|_2\le L\|\Delta x^k\|_2$.
By \eqref{eq:identical_gradient_increments}, the constant gradient offsets cancel when subtracting the gradient increments of the two instances.

At an odd iteration $k$ following the observation-matching argument in
\cite[Appendix J]{ding2022state}, we have $W\widetilde x^k=W\widetilde x^{\prime k}$ and $W\widetilde y^k=W\widetilde y^{\prime k}$. Subtracting the communication updates therefore yields
$
\Delta x^{k+1}=-\alpha\Delta y^k,
\,
\Delta y^{k+1}=h^{k+1}-h^k.
$
Substituting these relations into the subsequent local update gives
\begin{align}
\Delta x^{k+2}
&=-\alpha\Delta y^k-\alpha(h^{k+1}-h^k),
\label{eq:matched_local_primal}\\
\Delta y^{k+2}
&=\Delta y^{k+1}+h^{k+2}-h^{k+1}
=h^{k+2}-h^k.
\label{eq:matched_local_direction}
\end{align}

To bound the primal variable difference, we apply the triangle inequality to \eqref{eq:matched_local_primal} and use
$\|h^{k+1}\|_2\le L\|\Delta x^{k+1}\|_2
=\alpha L\|\Delta y^k\|_2$. We obtain
\begin{align}
\|\Delta x^{k+2}\|_2
&\le\alpha\|\Delta y^k\|_2
+\alpha\|h^{k+1}\|_2+\alpha\|h^k\|_2
\nonumber\\
&\le\alpha L\|\Delta x^k\|_2
+\alpha(1+\alpha L)\|\Delta y^k\|_2.
\label{eq:primal_sensitivity_bound}
\end{align}
Similarly, \eqref{eq:matched_local_direction} implies that
$
\|\Delta y^{k+2}\|_2
\le L\|\Delta x^{k+2}\|_2+L\|\Delta x^k\|_2.
$
Applying \eqref{eq:primal_sensitivity_bound} to the term $\|\Delta x^{k+2}\|_2$ yields
\begin{align}
\|\Delta y^{k+2}\|_2
&\le L(1+\alpha L)\|\Delta x^k\|_2
\nonumber\\
&\quad+\alpha L(1+\alpha L)\|\Delta y^k\|_2.
\label{eq:direction_sensitivity_bound}
\end{align}

From the coefficients in \eqref{eq:primal_sensitivity_bound} and \eqref{eq:direction_sensitivity_bound}, we define
\begin{equation}\label{eq:privacy_sensitivity_matrix}
B(\alpha):=
\begin{bmatrix}
\alpha L&\alpha(1+\alpha L)\\
L(1+\alpha L)&\alpha L(1+\alpha L)
\end{bmatrix}.
\end{equation}
For $s^k:=(\|\Delta x^k\|_2,\|\Delta y^k\|_2)^{\rm T}$,
the two-step sensitivity recursion is
$
s^{k+2}\le B(\alpha)s^k,
\; k\in\mathcal C,
$
where $B(\alpha)$ is component wise nonnegative matrix.

To initialize this recursion, we consider the first local update. Since the initial primal states coincide and the trackers are initialized with their respective gradients,
$\Delta x^0=0$ and $\Delta y^0=-a_i$. Hence,
$
\|\Delta y^0\|_2\le\delta_i,
\quad
\|\Delta x^1\|_2
=\alpha\|\Delta y^0\|_2\le\alpha\delta_i.
$
The local tracker update gives
$y_i^1=\nabla f_i(x_i^1)$ and
$y_i^{\prime1}=\nabla f_i'(x_i^{\prime1})$.
Using Lipschitz continuity and
$\nabla f_i'(x)=\nabla f_i(x)+a_i$, we obtain
$
\|\Delta y^1\|_2
\le L\|\Delta x^1\|_2+\|a_i\|_2
\le(1+\alpha L)\delta_i.
$
Thus, $s^1\le\delta_i b_\alpha$, where
\[
b_\alpha:=\begin{bmatrix}\alpha\\1+\alpha L\end{bmatrix}.
\]
Since $B(\alpha)$ is nonnegative, repeated applications of the recursion yield
\begin{equation}\label{eq:communication_sensitivity}
s^k\le\delta_i B(\alpha)^{(k-1)/2}b_\alpha,
\quad k\in\mathcal C.
\end{equation}
This bound holds uniformly over all possible observation sequences shared by the two adjacent problem instances.

\subsection{Infinite-Horizon Differential Privacy}

\begin{theorem}\label{thm:infinite_horizon_dp}
Suppose Assumption~\ref{ass:function} and the condition in~\eqref{eq:identical_gradient_increments} hold,  and all coordinates of the two communication perturbations are independent Laplace noise with the scales specified in Algorithm~\ref{alg:private}.
If
\begin{equation}\label{eq:privacy_decay_condition}
\rho(B(\alpha))<q_i^2<1,
\end{equation}
then the complete communication transcript is $\epsilon_i$ differentially private for changes at agent $i$, where
\begin{equation}\label{eq:privacy_matrix_budget}
\epsilon_i
=\sqrt d\,\delta_i
\begin{bmatrix}d_{x,i}^{-1}&d_{y,i}^{-1}\end{bmatrix}
\left(I-\frac{B(\alpha)}{q_i^2}\right)^{-1}b_\alpha.
\end{equation}
\end{theorem}

\begin{proof}
Fix adjacent problem instances $\mathcal P$ and $\mathcal P'$ that differ at agent $i$ according to Defintion~\ref{def:privacy_adjacency}.
We first establish the privacy inequality for a finite communication transcript and then extend it to the complete transcript in Definition~\ref{def:dp}.

Let $\mathcal M_K(\mathcal P)$ denote the transcript through an odd iteration $K$, and let $p_{\mathcal P}^{K}(o)$ denote its joint density at a possible transcript $o$. Define $p_{\mathcal P'}^{K}(o)$ similarly. When the preceding messages are fixed, the initialization and update rules determine the current internal states under each instance. The fresh perturbations are independent of those messages and have strictly positive Laplace densities.

For a scalar message $z=u+\eta$, where
$\eta\sim\operatorname{Lap}(b)$, its density is
$p(z\mid u)=(2b)^{-1}\exp(-|z-u|/b)$.
Thus, for two internal values $u$ and $v$ and the same scale $b$,
the triangle inequality gives
\[
\frac{p(z\mid u)}{p(z\mid v)}
=\exp\!\left(\frac{|z-v|-|z-u|}{b}\right)
\le\exp\!\left(\frac{|u-v|}{b}\right).
\]

Applying this inequality to each coordinate of the two messages transmitted by agent $i$ bounds the conditional log-density ratio at communication iteration $k$ by
\[
\frac{\|\Delta x^k\|_1}{d_{x,i}q_i^{k-1}}
+\frac{\|\Delta y^k\|_1}{d_{y,i}q_i^{k-1}}.
\]
Under the same preceding transcript, every other agent has identical internal states in the two instances. Its conditional message densities therefore cancel in the ratio. The joint transcript density is the product of the conditional message densities over communication iterations. Taking logarithms results in
\[
\log\frac{p_{\mathcal P}^{K}(o)}
              {p_{\mathcal P'}^{K}(o)}
\le
\sum_{\substack{k\in\mathcal C\\k\le K}}
\left(
\frac{\|\Delta x^k\|_1}{d_{x,i}q_i^{k-1}}
+\frac{\|\Delta y^k\|_1}{d_{y,i}q_i^{k-1}}
\right).
\]
Using $\|v\|_1\le\sqrt d\,\|v\|_2$ and the sensitivity bound \eqref{eq:communication_sensitivity}, we obtain
\begin{align*}
\log\frac{p_{\mathcal P}^{K}(o)}
              {p_{\mathcal P'}^{K}(o)}
\le{}&\sqrt d\,\delta_i
\begin{bmatrix}d_{x,i}^{-1}&d_{y,i}^{-1}\end{bmatrix}\\
&\times
\sum_{\substack{k\in\mathcal C\\k\le K}}
\left(\frac{B(\alpha)}{q_i^2}\right)^{(k-1)/2}b_\alpha.
\end{align*}
Since $B(\alpha)$ is nonnegative and $\rho(B(\alpha)/q_i^2)<1$, the matrix geometric series converges. Hence,
\[
\sum_{\substack{k\in\mathcal C\\k\le K}}
\left(\frac{B(\alpha)}{q_i^2}\right)^{(k-1)/2}b_\alpha
\le
\left(I-\frac{B(\alpha)}{q_i^2}\right)^{-1}b_\alpha
\]
componentwise. By \eqref{eq:privacy_matrix_budget}, this yields
$
p_{\mathcal P}^{K}(o)
\le e^{\epsilon_i}p_{\mathcal P'}^{K}(o)
$
for every finite transcript $o$, with a bound independent
of $K$.

Integrating over any measurable set $\mathcal O_K$ of finite transcripts gives
$
\Pr[\mathcal M_K(\mathcal P)\in\mathcal O_K]
\le e^{\epsilon_i}
\Pr[\mathcal M_K(\mathcal P')\in\mathcal O_K].
$
Thus, the required privacy inequality holds for every event determined by finitely many communication rounds.

The same bound holds for every finite transcript, independently of its length, and therefore extends to the complete infinite transcript. Hence, the algorithm is $\epsilon_i$-differentially private for changes at agent $i$.
\end{proof}

\subsection{Joint Parameter Conditions and Interpretation}

For $t=\alpha L>0$, the spectral radius of the sensitivity matrix is
\[
\rho(B(\alpha))
=\frac{2t+t^2+\sqrt{(2t+t^2)^2+4t(1+t)}}{2}.
\]
In particular,
$
0<\alpha L<\frac{\sqrt{17}-3}{4}
$
is equivalent to $\rho(B(\alpha))<1$ and ensures that an admissible
$q_i\in(\sqrt{\rho(B(\alpha))},1)$ exists. Indeed, the characteristic polynomial is
$
z^2-(2t+t^2)z-t(1+t),
$
whose positive root is less than one when $1-3t-2t^2>0$. Hence, a joint sufficient parameter choice is
\[
\begin{gathered}
0<\alpha<\min\left\{\frac{2\mu}{L^2},
\frac{\sqrt{17}-3}{4L}\right\},\, \rho(R(\alpha))<1,\\
\rho(B(\alpha))<q_i^2<1,\, d_{x,i}>0,\, d_{y,i}>0
\quad(i\in\mathcal V).
\end{gathered}
\]

These conditions are sufficient to guarantee both convergence
and differential privacy.




\section{Simulation Studies}
\label{sec:simulations}
We compare the proposed algorithm with method~\cite{ding2022state} on a five-agent ring. Following the sensing example in~\cite{ding2022state}, we consider $f_i(x)=\|M_i x-v_i\|^2$, where $x\in\mathbb R^{10}$ and $M_i\in\mathbb R^{10\times10}$ is a randomly generated nonsingular measurement matrix. The measurement vector is given by $v_i=M_i x_{\mathrm{true}}+e_i$, where $x_{\mathrm{true}}$ is the unknown parameter and $e_i$ is measurement noise.

We calibrate the noise scales of the two methods separately to give both methods the same certified infinite-horizon privacy budget $\epsilon_i=1$, with adjacency radius $\delta_i=0.01$, and display the first 700 communication rounds. For the proposed method, the noise scales decay as $(0.9)^{k-1}$ at communication iterations. Method~\cite{ding2022state} uses a decay factor of $0.9$ per communication round.

We consider two protocols: a common constant stepsize $\alpha=0.0031$, and separately tuned stepsizes selected by grid search followed by local refinement to minimize the expected MSE. The tuned stepsizes are approximately $0.0031$ for the proposed method and $0.0057$ for method in~\cite{ding2022state}. All candidates satisfy the corresponding sufficient convergence and privacy conditions, and noise scales are recalibrated for each candidate.

Performance is measured by
$
\mathcal E(r)
=
\frac{1}{5}\sum_{i=1}^{5}
\mathbb E\!\left[\|x_i(r)-x^\star\|_2^2\right],
$
where $x_i(r)$ is the state after communication round $r$ and the expectation is over the privacy perturbations. Figure~\ref{fig:sim_eps1} shows the exact expected MSE for the quadratic problem, validated using 500 Monte Carlo runs.

As shown in Figure~\ref{fig:sim_eps1}, The proposed method reduces the MSE faster per communication round and achieves lower MSE under both stepsize protocols at equal certified privacy budgets in the tested setting. All stepsizes accuracy satisfy the corresponding sufficient theoretical conditions.
\begin{figure}[t]
    \centering
    \includegraphics[width=\columnwidth]{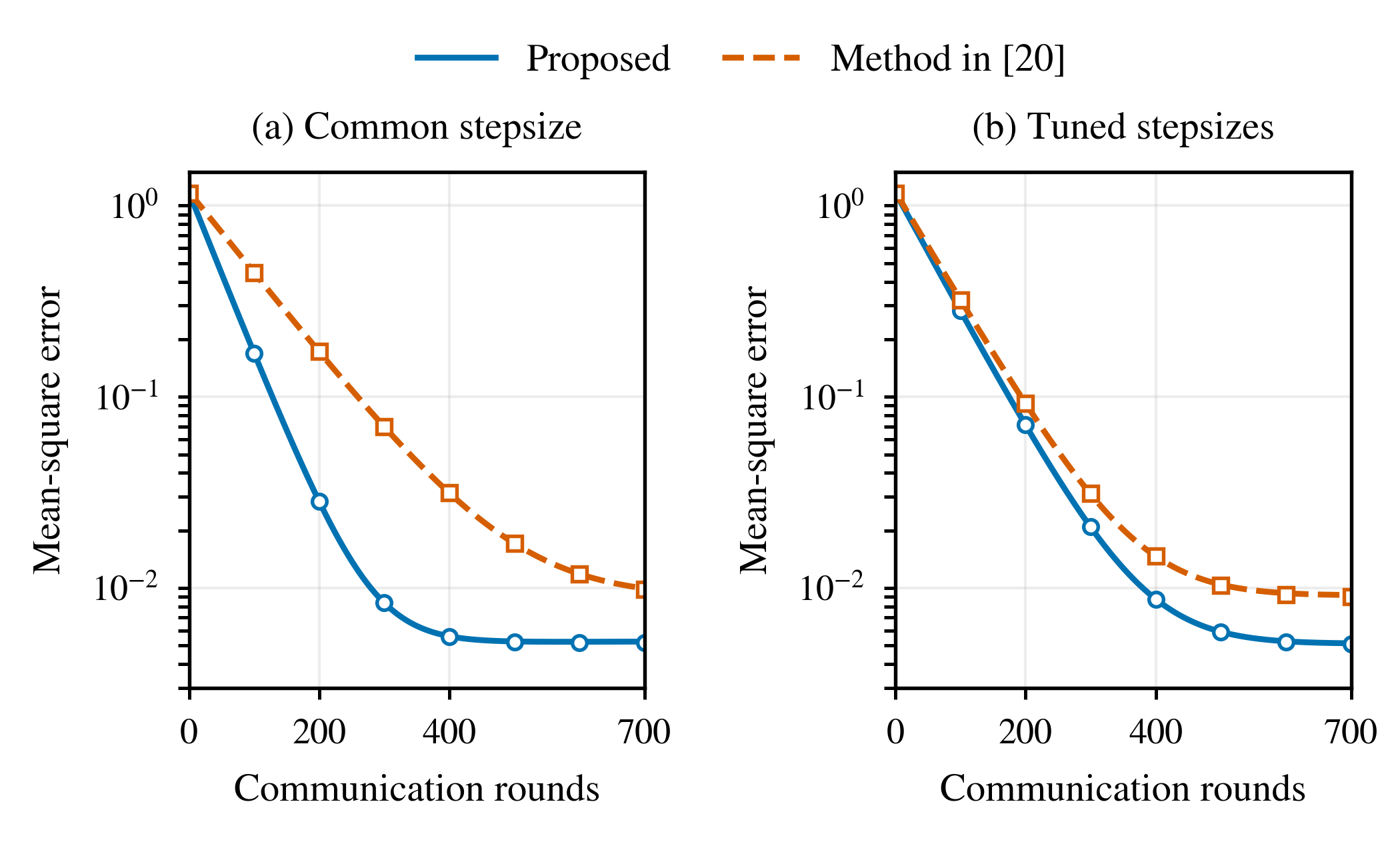}
    \caption{MSE comparison at equal infinite-horizon privacy budget: (a) common stepsize; (b) separately tuned admissible stepsizes. Lines show expected MSE; markers show averages of 500 Monte Carlo runs.}
    \label{fig:sim_eps1}
\end{figure}

\section{Conclusions}
We proposed a privacy-preserving gradient-tracking algorithm with one local update between communication rounds. By perturbing the communicated states and tracking directions, the method achieves infinite-horizon differential privacy under the stated adjacency relation. We established convergence in mean to a random consensus point
under sufficient parameter conditions. Simulations compared differentially private gradient tracking with and without local updates, illustrating the benefits of an additional local update under the same certified privacy budget.
\bibliographystyle{IEEEtran}
\bibliography{reference}

\end{document}